\documentclass[12pt, a4paper]{article}
\usepackage{}

\usepackage{amsfonts}
\usepackage{bbm}
\usepackage{mathrsfs}
\usepackage{latexsym}
\usepackage{graphicx}
\usepackage{amssymb}
\usepackage{amsmath}
\usepackage{color,xcolor}
\usepackage{mathtools}
\usepackage{cite}
\newtheorem{theorem}{Theorem}[section]
\newtheorem{lemma}[theorem]{Lemma}
\newtheorem{corollary}[theorem]{Corollary}

\newtheorem{remark}[theorem]{Remark}

\title{{\Large \bf  On the first Banhatti-Sombor index
\thanks{ Supported by the National Natural Science Foundation of China (Nos. 12071411 and 11771443).}~}}
\author{Zhen Lin$^{a}$\thanks{Corresponding author. E-mail addresses: lnlinzhen@163.com(Z. Lin), tb19080009b1@cumt.edu.cn(T. Zhou), vrkulli@gmail.com (V.R. Kulli), miaolianying@cumt.edu.cn (L. Miao).}, Ting Zhou$^b$, V.R. Kulli$^c$, Lianying Miao$^b$\\
{\footnotesize $^a$School of Mathematics and Statistics, Beijing Key Laboratory on MCAACI,}\\ {\footnotesize  Beijing Institute of Technology,}
\\  {\footnotesize  Beijing, 102488, P.R. China}\\
\footnotesize  $^b$School of Mathematics, China University of Mining and Technology,\\ \footnotesize Xuzhou, 221116, Jiangsu, P.R.
China \\
\footnotesize  $^c$Department of Mathematics, Gulbarga University,\\ \footnotesize Kalaburgi (Gulbarga)-585106, India
}

\date{}
\begin{document}
\openup 1.0\jot
\date{}\maketitle
\begin{abstract}
Let $d_v$ be the degree of the vertex $v$ in a connected graph $G$. The first Banhatti-Sombor index of $G$ is defined as $BSO(G) =\sum_{uv\in E(G)}\sqrt{\frac{1}{d^2_u}+\frac{1}{d^2_v}}$, which is a new vertex-degree-based topological index introduced by Kulli. In this paper, the
mathematical relations between the first Banhatti-Sombor index and some other well-known vertex-degree-based topological indices are established. In addition, the trees extremal with respect to the first Banhatti-Sombor index on trees and chemical trees are characterized, respectively.

\bigskip


\end{abstract}
\baselineskip 20pt

\section{\large Introduction}
Let $G$ be a simple undirected connected graph with vertex set $V(G)$ and edge set $E(G)$. The number of vertices and edges of $G$ is called order and size, respectively. Denote by $\overline{G}$ the complement of $G$. For $v\in V(G)$, $d_v$ denotes the degree of vertex $v$ in $G$. The minimum and the maximum degree of $G$ are denoted by $\delta(G)$ and $\Delta(G)$, or simply $\delta$ and $\Delta$, respectively. A pendant vertex of $G$ is a vertex of degree $1$. A graph $G$ is called $(\Delta, \delta)$-semiregular if $\{d_u, d_v\} = \{\Delta, \delta\}$ holds for all edges $uv\in E(G)$. Denote by $K_n$, $C_n$, $P_n$ and $K_{1,\,n-1}$ the complete graph, cycle, path and star with $n$ vertices, respectively.

The study of topological indices of various graph structures has been of interest to chemists, mathematicians, and scientists from related fields due to the fact that the topological indices play a significant role in mathematical chemistry especially in the QSPR/QSAR modeling.
In 1975, the Randi\'{c} index of a graph $G$ introduced by Randi\'{c} \cite{R} is the most important and widely applied. It is defined as
$$R(G)=\sum\limits_{uv\in E(G)}\frac{1}{\sqrt{d_ud_v}}.$$

The modified second Zagreb index of a graph $G$, introduced by Nikoli\'{c} et al. \cite{NKMT}, is defined as
$$M_2^{*}(G)=\sum\limits_{uv\in E(G)}\frac{1}{d_ud_v}$$

The harmonic index and the inverse degree index of a graph $G$ proposed by Fajtlowicz \cite{F} are two the older vertex-degree-based topological indices.
They are respectively defined as
$$H(G)=\sum\limits_{uv\in E(G)}\frac{2}{d_u+d_v}, \quad ID(G)=\sum\limits_{uv\in E(G)}\left(\frac{1}{d_u^2}+\frac{1}{d_v^2}\right).$$

The symmetric division deg index, the inverse sum indeg index and the geometric-arithmetic index of a graph $G$, introduced by Vuki\v{c}evi\'{c}
\cite{V, VG, VF}, Ga\v{s}perov \cite{VG} and Furtula \cite{VF}, are respectively defined as
$$SDD(G)=\sum\limits_{uv\in E(G)}\frac{d_u^2+d_v^2}{2d_ud_v}, \quad ISI(G)=\sum\limits_{uv\in E(G)}\frac{d_ud_v}{d_u+d_v}, \quad GA(G)=\sum\limits_{uv\in E(G)}\frac{2\sqrt{d_ud_v}}{d_u+d_v}.$$
The forgotten topological index, introduced by Furtula and Gutman \cite{FG}, is defined as
$$F(G)=\sum\limits_{uv\in E(G)}\left(d_u^2+d_v^2\right)$$

In 2021, the Sombor index of a graph $G$ is defined as
$$SO(G) =\sum_{uv\in E(G)}\sqrt{d_u^2+d_v^2},$$
which is a novel vertex-degree-based molecular structure descriptor proposed by Gutman \cite{G}. The investigation of the Sombor index of graphs has quickly received much attention. In particular, Red\v{z}epovi\'{c} \cite{R2} showed that the Sombor index may be used successfully on modeling thermodynamic properties of compounds due to the fact that the Sombor index has satisfactory prediction potential in modeling entropy and enthalpy of vaporization of alkanes. Das et al. \cite{DCC} and Wang et al. \cite{WMLF} gave the mathematical relations between the Sombor index and some other well-known vertex-degree-based topological indices. For other related results, one may refer to \cite{CGR, DTW, LMZ, KG, MMM, RDA} and the references therein.

Inspired by work on Sombor index, the first Banhatti-Sombor index of a connected graph $G$ was introduced by Kulli \cite{K} very recently and is defined as
$$BSO(G) =\sum_{uv\in E(G)}\sqrt{\frac{1}{d^2_u}+\frac{1}{d^2_v}}.$$
We find that the new index  has close contact with numerous well-known vertex-degree-based topological indices. Moreover, the trees with the maximum and minimum first Banhatti-Sombor index among the set of trees with $n$ vertices are determined, respectively. In particular, the extremal values of the first Banhatti-Sombor index for chemical trees are characterized.

\section{\large  Preliminaries}

\begin{lemma}\label{le2,1} 
For any edge $uv\in E(G)$, $d_u^2+d_v^2$ or $\frac{1}{d_u^2}+\frac{1}{d_v^2}$ is a constant if and only if $G$ is a regular graph (when $G$ is non-bipartite) or $G$ is a $(\Delta, \delta)$-semiregular bipartite graph (when $G$ is bipartite).
\end{lemma}

\begin{lemma}\label{le2,2} 
For any positive real number $a$ and $b$, we have
$$\frac{2\sqrt{2}(a^2+b^2+ab)}{3(a+b)}\leq \sqrt{a^2+b^2}\leq\frac{\sqrt{2}(a^2+b^2)}{a+b}$$
with equality if and only if $a=b$.
\end{lemma}

\begin{lemma}{\bf (\cite{R1})}\label{le2,3} 
If $a_i>0$, $b_i>0$, $p>0$, $i=1, 2, \ldots, n$, then the following inequality holds:
$$\sum\limits_{i=1}^{n}\frac{a_k^{p+1}}{b_k^p}\geq \frac{\left(\sum\limits_{i=1}^{n}a_i\right)^{p+1}}{\left(\sum\limits_{i=1}^{n}b_i\right)^{p}}$$
with equality if and only if $\frac{a_1}{b_1}=\frac{a_2}{b_2}=\cdots=\frac{a_n}{b_n}$.
\end{lemma}

\begin{lemma}{\bf (\cite{DM})}\label{le2,4} 
Let $a_1, a_2, \ldots, a_n$ and $b_1, b_2, \ldots, b_n$ be real numbers such
that $q\leq \frac{a_i}{b_i}\leq Q$ and $a_i\neq 0$ for $i=1, 2, \ldots, n$. Then there holds
$$\sum\limits_{i=1}^{n}b_i^2+Qq\sum\limits_{i=1}^{n}a_i^2\leq (Q+q)\sum\limits_{i=1}^{n}a_ib_i$$
with equality if and only if $b_i=qa_i$ or $b_i=Qa_i$ for et least one $i$, $i=1, 2, \ldots, n$.
\end{lemma}

\begin{lemma}{\bf (\cite{D})}\label{le2,5} 
If $a=(a_1, a_2, \ldots, a_n)$, $b=(b_1, b_2, \ldots, b_n)$ are sequences of real numbers and
$c=(c_1, c_2, \ldots, c_n)$, $d=(d_1, d_2, \ldots, d_n)$ are nonnegative, then
$$\sum\limits_{i=1}^{n}d_i\sum\limits_{i=1}^{n}c_ia_i^2+\sum\limits_{i=1}^{n}c_i\sum\limits_{i=1}^{n}d_ib_i^2\geq 2\sum\limits_{i=1}^{n}c_ia_i\sum\limits_{i=1}^{n}d_ib_i$$
with equality if and only if $a=b=(k, k, \ldots, k)$ is a constant sequence for positive $c_i$ and $d_i$, $i=1, 2, \ldots, n$.
\end{lemma}

\begin{lemma}{\bf (\cite{LMR})}\label{le2,6} 
Let $a_1, a_2, \ldots, a_n$ and $b_1, b_2, \ldots, b_n$ be real numbers such
that $a\leq a_i\leq A$ and $b\leq b_i\leq B$ for $i=1, 2, \ldots, n$. Then there holds
$$\left\lvert \frac{1}{n}\sum\limits_{i=1}^{n}a_ib_i-\frac{1}{n}\sum\limits_{i=1}^{n}a_i\frac{1}{n}\sum\limits_{i=1}^{n}b_i\right\rvert\leq \frac{1}{n}\left\lfloor\frac{n}{2}\right\rfloor\left(1-\frac{1}{n}\left\lfloor\frac{n}{2}\right\rfloor\right)(A-a)(B-b),$$
where $\lfloor x\rfloor$ denotes the integer part of $x$.
\end{lemma}

\section{\large  On relations between the first Banhatti-Sombor index and other degree-based indices}

\subsection{\large  Bounds in terms of order, size and degree}

\begin{theorem}\label{th3,1} 
Let $G$ be a connected graph of order $n$ and size $m$ with the minimum degree $\delta$. Then
$$\frac{n}{\sqrt{2}}\leq BSO(G)\leq \frac{\sqrt{2}m}{\delta}$$
with equality if and only if $G$ is a regular graph.
\end{theorem}

\begin{proof} Note that
$$BSO(G)=\sum\limits_{uv\in E(G)}\sqrt{\frac{1}{d_u^2}+\frac{1}{d_v^2}}\leq \sum\limits_{uv\in E(G)}\sqrt{\frac{1}{\delta^2}+\frac{1}{\delta^2}}=\frac{\sqrt{2}m}{\delta}$$
with equality if and only if $d_u=\delta$ for any vertex $u$, that is, $G$ is a regular graph.

By the Cauchy-Schwarz inequality, we have
$$BSO(G)=\sum\limits_{uv\in E(G)}\sqrt{\frac{1}{d_u^2}+\frac{1}{d_v^2}}\geq \sum\limits_{uv\in E(G)}\frac{1}{\sqrt{2}}\left(\frac{1}{d_u}+\frac{1}{d_v}\right)=\frac{1}{\sqrt{2}}n$$
with equality if and only if $d_u=d_v$ for any edge $uv$, that is, $G$ is a regular graph.

\end{proof}

\begin{corollary}\label{cor3,1} 
Let $G$ be a regular connected graph with $n$ vertices. Then
$$BSO(G)= \frac{n}{\sqrt{2}}.$$
\end{corollary}

\begin{remark}
This implies that $BSO(G)$ dose not increase with the increase of the number of edges of $G$. Clearly, $BSO(K_n)=BSO(C_n)$.
\end{remark}

\begin{corollary}\label{cor3,2} 
Let $U_n$ be a unicyclic graph with $n$ vertices. Then
$$BSO(U_n)\geq \frac{n}{\sqrt{2}}$$
with equality if and only if $G\cong C_n$.
\end{corollary}

\begin{corollary}\label{cor3,3} 
Let $G$ be a connected graph of order $n$ and size $m$ with the maximum degree $\Delta$ and the minimum degree $\delta$. Then
$$\sqrt{2} n\leq BSO(G)+BSO(\overline{G})\leq \sqrt{2}\left(\frac{m}{\delta}+\frac{n(n-1)-2m}{2(n-1-\Delta)}\right)$$
with equality if and only if $G$ is a regular graph.
\end{corollary}

\begin{theorem}\label{th3,2} 
Let $G$ be a connected graph of order $n$ and size $m$ with the maximum degree $\Delta$. Then
$$BSO(G)\leq n-m(2-\sqrt{2})\frac{1}{\Delta}$$
with equality if and only if $G$ is a regular graph.
\end{theorem}

\begin{proof}Without loss of generality, we suppose that $d_u\geq d_v$. Then we have
\begin{eqnarray*}
BSO(G) & = & \sum\limits_{uv\in E(G)}\sqrt{\frac{1}{d_u^2}+\frac{1}{d_v^2}}
 \leq  \sum\limits_{uv\in E(G)}\left(\frac{1}{d_v}+(\sqrt{2}-1)\frac{1}{d_u}\right)\\
& \leq & \sum\limits_{uv\in E(G)}\left(\frac{1}{d_v}+\frac{1}{d_u}\right)+m(\sqrt{2}-2)\frac{1}{\Delta}
 = n-m(2-\sqrt{2})\frac{1}{\Delta}
\end{eqnarray*}
with equality if and only if $G$ is a regular graph.
\end{proof}

\begin{corollary}\label{cor3,4} 
Let $G$ be a connected graph of order $n$ and size $m$ with the maximum degree $\Delta$ and the minimum degree $\delta$. Then
$$\sqrt{2} n\leq BSO(G)+BSO(\overline{G})\leq 2n-(2-\sqrt{2})\left(\frac{m}{\Delta}+\frac{n(n-1)-2m}{2(n-1-\delta)}\right)$$
with equality if and only if $G$ is a regular graph.
\end{corollary}

\subsection{\large  Bounds in terms of the Randi\'{c} index,  the modified second Zagreb index and the inverse degree index}

\begin{theorem}\label{th3,3} 
Let $G$ be a connected graph with the maximum degree $\Delta$. Then
$$\sqrt{2}R(G)\leq BSO(G)\leq \sqrt{2}\Delta M_2^{*}(G)$$
with equality if and only if $G$ is a regular graph.
\end{theorem}

\begin{proof} By the arithmetic geometric inequality, we have
$$BSO(G)=\sum\limits_{uv\in E(G)}\sqrt{\frac{1}{d_u^2}+\frac{1}{d_v^2}}\geq \sum\limits_{uv\in E(G)}\sqrt{\frac{2}{d_ud_v}}=\sqrt{2}R(G)$$
with equality if and only if $d_u=d_v$ for any edges, that is, $G$ is a regular graph.
It easy to see that
$$BSO(G)=\sum\limits_{uv\in E(G)}\sqrt{\frac{1}{d_u^2}+\frac{1}{d_v^2}}\leq \sum\limits_{uv\in E(G)}\frac{\sqrt{2\Delta^2}}{d_ud_v}=\sqrt{2}\Delta M_2^{*}(G)$$
with equality if and only if $d_u=d_v=\Delta$ for any edges, that is, $G$ is a regular graph.
\end{proof}

\begin{theorem}\label{th3,4} 
Let $G$ be a connected graph with the maximum degree $\Delta$ and the minimum degree $\delta$. Then
$$BSO(G)\leq \sqrt{mID(G)}$$
with equality if and only if $G$ is a regular graph (when $G$ is non-bipartite) or $G$ is a $(\Delta, \delta)$-semiregular bipartite graph (when $G$ is bipartite).
\end{theorem}

\begin{proof}By the Cauchy-Schwarz inequality, we have
$$BSO(G)= \sum\limits_{uv\in E(G)}1\cdot \sqrt{\frac{1}{d_u^2}+\frac{1}{d_v^2}}\leq \sqrt{\sum\limits_{uv\in E(G)}1^2\sum\limits_{uv\in E(G)}\left(\frac{1}{d_u^2}+\frac{1}{d_v^2}\right)}=\sqrt{mID(G)},$$
with equality if and only if $\frac{1}{d_u^2}+\frac{1}{d_v^2}$ is a constant for any edge $uv$ in a connected graph $G$. By Lemma \ref{le2,1}, $G$ is a regular graph (when $G$ is non-bipartite) or $G$ is a $(\Delta, \delta)$-semiregular bipartite graph (when $G$ is bipartite).
\end{proof}

\subsection{\large  Bounds in terms of the harmonic index, the symmetric division deg index and the modified second Zagreb index}

\begin{theorem}\label{th3,5} 
Let $G$ be a connected graph with the maximum degree $\Delta$ and the minimum degree $\delta$. Then
$$\sqrt{2}H(G)\leq BSO(G)\leq  \frac{1}{\sqrt{2}}\left(\frac{\Delta}{\delta}+\frac{\delta}{\Delta}\right)H(G)$$
with equality if and only if $G$ is a regular graph.
\end{theorem}

\begin{proof} By Lemma \ref{le2,2}, we have
\begin{eqnarray*}
BSO(G) & = & \sum\limits_{uv\in E(G)}\sqrt{\frac{1}{d_u^2}+\frac{1}{d_v^2}}\leq  \sum\limits_{uv\in E(G)}\frac{\sqrt{2}\left(\frac{d_v}{d_u}+\frac{d_u}{d_v}\right)}{d_u+d_v}\\
& \leq & \sum\limits_{uv\in E(G)}\frac{1}{\sqrt{2}}\left(\frac{\Delta}{\delta}+\frac{\delta}{\Delta}\right)\frac{2}{d_u+d_v}
 =  \frac{1}{\sqrt{2}}\left(\frac{\Delta}{\delta}+\frac{\delta}{\Delta}\right)H(G).
\end{eqnarray*}
with equality if and only if $d_u=d_v$ for any edges, that is, $G$ is a regular graph.

By Lemma \ref{le2,2}, we have
\begin{eqnarray*}
BSO(G) & = & \sum\limits_{uv\in E(G)}\sqrt{\frac{1}{d_u^2}+\frac{1}{d_v^2}}\geq  \sum\limits_{uv\in E(G)}\frac{2\sqrt{2}\left(\frac{d_v}{d_u}+\frac{d_u}{d_v}+1\right)}{3(d_u+d_v)}\\
& \geq & \sum\limits_{uv\in E(G)}\frac{2\sqrt{2}(2+1)}{3(d_u+d_v)}=  \sqrt{2}H(G)
\end{eqnarray*}
with equality if and only if $d_u=d_v$ for any edges, that is, $G$ is a regular graph.
\end{proof}

\begin{theorem}\label{th3,6} 
Let $G$ be a connected graph with the maximum degree $\Delta$ and the minimum degree $\delta$. Then
$$\frac{2\sqrt{2}}{3\Delta}SDD(G)+\frac{\sqrt{2}}{3}H(G)\leq BSO(G)\leq  \frac{\sqrt{2}}{\delta}SDD(G)$$
with equality if and only if $G$ is a regular graph.
\end{theorem}

\begin{proof}By Lemma \ref{le2,2}, we have
\begin{eqnarray*}
BSO(G) & = & \sum\limits_{uv\in E(G)}\sqrt{\frac{1}{d_u^2}+\frac{1}{d_v^2}}\leq  \sum\limits_{uv\in E(G)}\frac{\sqrt{2}\left(\frac{d_v}{d_u}+\frac{d_u}{d_v}\right)}{d_u+d_v}\\
& \leq & \sum\limits_{uv\in E(G)}\frac{\sqrt{2}\left(\frac{d_v}{d_u}+\frac{d_u}{d_v}\right)}{\delta+\delta}
 =  \frac{\sqrt{2}}{\delta}SDD(G).
\end{eqnarray*}
with equality if and only if $d_u=d_v=\delta$ for any edges, that is, $G$ is a regular graph.

By Lemma \ref{le2,2}, we have
\begin{eqnarray*}
BSO(G) & = & \sum\limits_{uv\in E(G)}\sqrt{\frac{1}{d_u^2}+\frac{1}{d_v^2}}
\geq  \sum\limits_{uv\in E(G)}\frac{2\sqrt{2}\left(\frac{d_v}{d_u}+\frac{d_u}{d_v}+1\right)}{3(d_u+d_v)}\\
& = & \sum\limits_{uv\in E(G)}\frac{2\sqrt{2}\left(\frac{d_v}{d_u}+\frac{d_u}{d_v}\right)}{3(d_u+d_v)}+\sum\limits_{uv\in E(G)}\frac{2\sqrt{2}}{3(d_u+d_v)}
 =  \frac{2\sqrt{2}}{3\Delta}SDD(G)+\frac{\sqrt{2}}{3}H(G).
\end{eqnarray*}
with equality if and only if $d_u=d_v=\Delta$ for any edges, that is, $G$ is a regular graph.
\end{proof}

\begin{theorem}\label{th3,7} 
Let $G$ be a connected graph with $n$ vertices. Then
$$BSO(G)\leq \sqrt{2M_2^{*}(G)SDD(G)}$$
with equality if and only if $G$ is a regular graph (when $G$ is non-bipartite) or $G$ is a $(\Delta, \delta)$-semiregular bipartite graph (when $G$ is bipartite).
\end{theorem}

\begin{proof} Let $r=1$, $a_i=\sqrt{\frac{1}{d_u^2}+\frac{1}{d_v^2}}$ and $b_i=\frac{1}{d_ud_v}$ in Lemma \ref{le2,3}. Then we have
$$\frac{\left(\sum\limits_{uv\in E(G)}\sqrt{\frac{1}{d_u^2}+\frac{1}{d_v^2}}\right)^2}{\sum\limits_{uv\in E(G)}\frac{1}{d_ud_v}}\leq \sum\limits_{uv\in E(G)}\frac{\frac{1}{d_u^2}+\frac{1}{d_v^2}}{\frac{1}{d_ud_v}}=\sum\limits_{uv\in E(G)}\left(\frac{d_v}{d_u}+\frac{d_u}{d_v}\right),$$
that is,
$$BSO(G)\leq \sqrt{2M_2^{*}(G)SDD(G)}$$
with equality if and only if $\sqrt{d_u^2+d_v^2}$ is a constant for any edge $uv$ in $G$, by Lemma \ref{le2,1}, $G$ is a regular graph (when $G$ is non-bipartite) or $G$ is a $(\Delta, \delta)$-semiregular bipartite graph (when $G$ is bipartite).
\end{proof}

\begin{corollary}\label{cor3,5} 
Let $G$ be a connected graph of order $n$ and size $m$ with the maximum degree $\Delta$ and the minimum degree $\delta$. Then
$$BSO(G)\leq \sqrt{mM_2^{*}(G)\left(\frac{\Delta}{\delta}+\frac{\delta}{\Delta}\right)}$$
with equality if and only if $G$ is a regular graph or a $(\Delta, \delta)$-semiregular bipartite graph.
\end{corollary}

\begin{proof} Without loss of generality, we assume that $d_u\geq d_v$. By the proof of Theorem \ref{th3,7}, we have
$$\frac{\left(\sum\limits_{uv\in E(G)}\sqrt{\frac{1}{d_u^2}+\frac{1}{d_v^2}}\right)^2}{\sum\limits_{uv\in E(G)}\frac{1}{d_ud_v}}\leq \sum\limits_{uv\in E(G)}\left(\frac{d_v}{d_u}+\frac{d_u}{d_v}\right)\leq \left(\frac{\Delta}{\delta}+\frac{\delta}{\Delta}\right)m$$
with equality if and only if $d_u=\Delta$ and $d_v=\delta$ for any edge $uv$. This implies that $G$ is $G$ is a regular graph or a $(\Delta, \delta)$-semiregular bipartite graph. Conversely, it is easy to check that equality holds in Corollary \ref{cor3,4} when $G$ is a regular graph or a $(\Delta, \delta)$-semiregular bipartite graph.
\end{proof}

\subsection{\large  Bounds in terms of the forgotten index}

\begin{theorem}\label{th3,8} 
Let $G$ be a connected graph of order $n$ and size $m$ with the maximum degree $\Delta$ and the minimum degree $\delta$. Then
$$BSO(G)\geq \frac{\sqrt{2}}{\Delta^3+\delta^3}\left(\frac{m\delta^3}{\Delta}+\frac{F}{2}\right)$$
with equality if and only if $G$ is a regular graph.
\end{theorem}

\begin{proof}Let $a_i=\sqrt{d_u^2+d_v^2}$ and $b_i=\frac{1}{d_ud_v}$ in Lemma \ref{le2,4}. Then $q=\frac{1}{\sqrt{2}\Delta^3}$ and $Q=\frac{1}{\sqrt{2}\delta^3}$. By Lemma \ref{le2,4}, we have
$$\sum\limits_{uv\in E(G)}\frac{1}{d_u^2d_v^2}+\frac{1}{2\Delta^3\delta^3}\sum\limits_{uv\in E(G)}(d_u^2+d_v^2)\leq \frac{1}{\sqrt{2}}\left(\frac{1}{\Delta^3}+\frac{1}{\delta^3}\right)\sum\limits_{uv\in E(G)}\sqrt{\frac{1}{d_u^2}+\frac{1}{d_v^2}},$$
that is,
$$\frac{m}{\Delta^4}+\frac{1}{2\Delta^3\delta^3}F(G)\leq \frac{1}{\sqrt{2}}\left(\frac{1}{\Delta^3}+\frac{1}{\delta^3}\right)BSO(G),$$
that is,
$$BSO(G)\geq \frac{\sqrt{2}}{\Delta^3+\delta^3}\left(\frac{m\delta^3}{\Delta}+\frac{F}{2}\right)$$
with equality if and only if $d_u=d_v=\Delta$ for any edge $uv$, that is, $G$ is a regular graph.
\end{proof}

\begin{theorem}\label{th3,9} 
Let $G$ be a connected graph of order $n$ and size $m$ with the maximum degree $\Delta$ and the minimum degree $\delta$. Then
$$BSO(G)\leq \frac{2mSDD(G)+M_2^{*}(G)F(G)}{2SO(G)}$$
with equality if and only if $G$ is a regular graph (when $G$ is non-bipartite) or $G$ is a $(\Delta, \delta)$-semiregular bipartite graph (when $G$ is bipartite).
\end{theorem}

\begin{proof} Let $a_i=b_i=\sqrt{d_u^2+d_v^2}$, $c_i=\frac{1}{d_ud_v}$ and $d_i=1$ in Lemma \ref{le2,5}. Then we have
$$m\sum\limits_{uv\in E(G)}\frac{d_u^2+d_v^2}{d_ud_v}+\sum\limits_{uv\in E(G)}\frac{1}{d_ud_v}\sum\limits_{uv\in E(G)}(d_u^2+d_v^2)\geq 2\sum\limits_{uv\in E(G)}\frac{\sqrt{d_u^2+d_v^2}}{d_ud_v}\sum\limits_{uv\in E(G)}\sqrt{d_u^2+d_v^2},$$
that is,
$$2mSDD(G)+M_2^{*}(G)F(G)\geq 2BSO(G)SO(G),$$
that is,
$$BSO(G)\leq \frac{2mSDD(G)+M_2^{*}(G)F(G)}{2SO(G)}$$
with equality if and only if $a_i=b_i=\sqrt{d_u^2+d_v^2}$ for any edge $uv$ in $G$, that is, $d_u^2+d_v^2$ is a constant for any edge $uv$ in $G$, by Lemma \ref{le2,1}, $G$ is a regular graph (when $G$ is non-bipartite) or $G$ is a $(\Delta, \delta)$-semiregular bipartite graph (when $G$ is bipartite).
\end{proof}

\begin{corollary}\label{cor3,6} 
Let $G$ be a connected graph of size $m$ with the maximum degree $\Delta$ and the minimum degree $\delta$. Then
$$BSO(G)\leq \frac{m(\Delta^2\delta+\delta^2)+\Delta F(G)}{2\sqrt{2}\Delta\delta^3}$$
with equality if and only if $G$ is a regular graph.
\end{corollary}

\begin{corollary}\label{cor3,7} 
Let $G$ be a connected graph of size $m$ with the maximum degree $\Delta$ and the minimum degree $\delta$. Then
$$BSO(G)\leq \frac{m^2(2\Delta^3+\Delta^2\delta+\delta^3)}{2\Delta\delta^2SO(G)}$$
with equality if and only if $G$ is a regular graph.
\end{corollary}

\subsection{\large  Bounds in terms of the inverse sum indeg index and the geometric-arithmetic index}

\begin{theorem}\label{th3,10} 
Let $G$ be a connected graph of order $n$ and size $m$ with the maximum degree $\Delta$ and the minimum degree $\delta$. Then
$$BSO(G)\leq \frac{H(G)SDD(G)+2M_2^{*}(G)ISI(G)}{\sqrt{2}GA(G)}$$
with equality if and only if $G$ is a regular graph.
\end{theorem}

\begin{proof} Let $a_i=\sqrt{d_u^2+d_v^2}$, $b_i=\sqrt{2d_ud_v}$, $c_i=\frac{1}{d_ud_v}$ and $d_i=\frac{1}{d_u+d_v}$ in Lemma \ref{le2,5}. Then we have
\begin{align*}
 & \sum\limits_{uv\in E(G)}\frac{1}{d_u+d_v}\sum\limits_{uv\in E(G)}\frac{d_u^2+d_v^2}{d_ud_v}+\sum\limits_{uv\in E(G)}\frac{1}{d_ud_v}\sum\limits_{uv\in E(G)}\frac{2d_ud_v}{d_u+d_v}\\
\geq {}& 2\sum\limits_{uv\in E(G)}\frac{\sqrt{d_u^2+d_v^2}}{d_ud_v}\sum\limits_{uv\in E(G)}\frac{\sqrt{2d_ud_v}}{d_u+d_v},
\end{align*}
that is,
$$H(G)SDD(G)+2M_2^{*}(G)ISI(G)\geq \sqrt{2}BSO(G)GA(G),$$
that is,
$$BSO(G)\leq \frac{H(G)SDD(G)+2M_2^{*}(G)ISI(G)}{\sqrt{2}GA(G)}$$
with equality if and only if $\sqrt{d_u^2+d_v^2}=\sqrt{2d_ud_v}$ for any edge $uv$, that is, $G$ is a regular graph.
\end{proof}

\begin{corollary}\label{cor3,8} 
Let $G$ be a connected graph of size $m$ with the maximum degree $\Delta$ and the minimum degree $\delta$. Then
$$BSO(G)\leq \frac{m^2\Delta^2+m^2\delta^2+4m\Delta ISI(G)}{2\sqrt{2}\Delta\delta^2GA(G)}$$
with equality if and only if $G$ is a regular graph.
\end{corollary}

\subsection{\large  Bounds in terms of the Sombor index and the modified second Zagreb index}

\begin{theorem}\label{th3,11} 
Let $G$ be a connected graph of size $m$ with the maximum degree $\Delta$ and the minimum degree $\delta$. Then
$$\frac{2m^2}{SO(G)}\leq BSO(G)\leq \frac{1}{\delta^2}SO(G)$$
with equality if and only if $G$ is a regular graph.
\end{theorem}

\begin{proof}It is easy to see that
$$BSO(G)=\sum\limits_{uv\in E(G)}\sqrt{\frac{1}{d_u^2}+\frac{1}{d_v^2}}\leq \frac{1}{\delta^2}\sum\limits_{uv\in E(G)}\sqrt{d_u^2+d_v^2}\leq\frac{1}{\delta^2}SO(G),$$
with equality if and only if $d_u=d_v=\Delta$ for any edges, that is, $G$ is a regular graph.

Let $a_i=b_i=\frac{1}{\sqrt{d_ud_v}}$ and $c_i=d_i=\sqrt{d_u^2+d_v^2}$ in Lemma \ref{le2,5}. Then
$$2\sum\limits_{uv\in E(G)}\sqrt{d_u^2+d_v^2}\sum\limits_{uv\in E(G)}\sqrt{\frac{1}{d_u^2}+\frac{1}{d_v^2}}\geq 2\left(\sum\limits_{uv\in E(G)}\sqrt{\frac{d_u^2+d_v^2}{d_ud_v}}\right)^2\geq 4m^2,$$
that is,
$$2SO(G)BSO(G)\geq 4m^2,$$
that is,
$$BSO(G)\geq \frac{2m^2}{SO(G)}$$
with equality if and only if $G$ is a regular graph.
\end{proof}

\begin{theorem}\label{th3,12} 
Let $G$ be a connected graph of order $n$ and size $m$ with the maximum degree $\Delta$ and the minimum degree $\delta$. Then
$$\left\lvert \frac{1}{m}BSO(G)-\frac{1}{m^2}SO(G)M_2^{*}(G)\right\rvert\leq \xi(m)\frac{\sqrt{2}(\Delta+\delta)
(\Delta-\delta)^2}{\Delta^2\delta^2},$$
where
$$\xi(m)=\frac{1}{4}\left(1-\frac{1+(-1)^{m+1}}{2m^2}\right).$$
\end{theorem}

\begin{proof} Let $a_i=\sqrt{d_u^2+d_v^2}$ and $b_i=\frac{1}{d_ud_v}$ in Lemma \ref{le2,6}. Then $a=\sqrt{2}\delta$, $A=\sqrt{2}\Delta$, $b=\frac{1}{\Delta^2}$ and $B=\frac{1}{\delta^2}$. By Lemma \ref{le2,6}, we have
\begin{align*}
 & \left\lvert \frac{1}{m}\sum\limits_{uv\in E(G)}\sqrt{\frac{1}{d_u^2}+\frac{1}{d_v^2}}-\frac{1}{m^2}\sum\limits_{uv\in E(G)}\sqrt{d_u^2+d_v^2}\sum\limits_{uv\in E(G)}\frac{1}{d_ud_v}\right\rvert \\
\leq {}& \frac{1}{m}\left\lfloor\frac{m}{2}\right\rfloor\left(1-\frac{1}{m}\left\lfloor\frac{m}{2}\right\rfloor\right)
\sqrt{2}(\Delta-\delta)(\frac{1}{\delta^2}-\frac{1}{\Delta^2}),
\end{align*}
that is,
$$\left\lvert \frac{1}{m}BSO(G)-\frac{1}{m^2}SO(G)M_2^{*}(G)\right\rvert\leq \xi(m)\frac{\sqrt{2}(\Delta+\delta)
(\Delta-\delta)^2}{\Delta^2\delta^2},$$
where
$$\xi(m)=\frac{1}{m}\left\lfloor\frac{m}{2}\right\rfloor\left(1-\frac{1}{m}\left\lfloor\frac{m}{2}\right\rfloor\right)
=\frac{1}{4}\left(1-\frac{1+(-1)^{m+1}}{2m^2}\right).$$
\end{proof}

\section{\large  The first Banhatti-Sombor index of trees}

In this section, we determine the trees with the maximum and minimum first Banhatti-Sombor index among the set of trees of order $n$, respectively.
For a tree $T_n$ of order $n$ with maximum degree $\Delta$, denote by $n_i$ the number of vertices with degree $i$ in $T_n$ for $1\leq i\leq\Delta$, and $m_{i,j}$ the number of edges of $T_n$ connecting vertices of degree $i$ and $j$, where $1\leq i\leq j\leq\Delta$. Note that $T_n$ is connected, so $m_{1,1}=0$ for $n\geq 3$. Let $N=\{(i,j)\in\mathbb{N}\times\mathbb{N}:1\leq i\leq j\leq\Delta\}$. Then clearly the following relations hold:

\begin{equation}\label{eq1}
|V(T_n)|=n=\sum_{i=1}^{\Delta}n_i,
\end{equation}

\begin{equation}\label{eq2}
|E(T_n)|=n-1=\sum_{(i,j)\in N}m_{i,j},
\end{equation}
and
\begin{equation}\label{eq3}
\left\{
\begin{aligned}
&2m_{1,1}+m_{1,2}+\ldots+m_{1,\Delta}=n_1,\\
&m_{1,2}+2m_{2,2}+\ldots+m_{2,\Delta}=2n_2,\\
&\ldots\\
&m_{1,\Delta}+m_{2,\Delta}+\ldots+2m_{\Delta,\Delta}=\Delta n_\Delta.
\end{aligned}
\right.
\end{equation}
It follows easily from (\ref{eq1}) and (\ref{eq3}) that
 \begin{equation}\label{eq4}
 n=\sum_{(i,j)\in N}\frac{i+j}{ij}m_{i,j}.
 \end{equation}
And the definition of the first Banhatti-Sombor index is equivalent to
 \begin{equation}\label{eq5}
 SO(G)=\sum_{(i,j)\in P}\sqrt{\frac{1}{i^2}+\frac{1}{j^2}}m_{i,j}.
 \end{equation}

\begin{theorem}\label{th4,1} 
Let $T_n$ be a tree with $n$-vertex. Then
$$\frac{\sqrt{2}(n-3)}{2}+\sqrt{5}\leq BSO(T_n)\leq \sqrt{1+(n-1)^2}.$$
The equality in the left hand side holds if and only if $T_n\cong P_n$, and the equality in the
right hand side holds if and only if $T_n\cong K_{1,\,n-1}$.
\end{theorem}

\begin{proof}
 First, we consider the equality in the left side. Let $N_1=\Big\{(i,j)\in N:(i,j)\neq(1,1),(i,j)\neq(1,2),(i,j)\neq(2,2)\Big\}$. By equation (\ref{eq4}), we have
$$3m_{1,2}+2m_{2,2}=2n-\sum_{(i,j)\in N_1}\frac{2(i+j)}{ij}m_{i,j},$$
and by equation (\ref{eq2}), we have
$$m_{1,2}+m_{2,2}=n-1-\sum_{(i,j)\in N_1}m_{i,j}.$$
Then we obtain the following expression for $m_{1,2}$ and $m_{2,2}$:
$$m_{1,2}=2+\sum_{(i,j)\in N_1}\Big[2-\frac{2(i+j)}{ij}\Big]m_{i,j},$$
$$m_{2,2}=n-3+\sum_{(i,j)\in N_1}\Big[\frac{2(i+j)}{ij}-3\Big]m_{i,j}.$$
According to the expression (\ref{eq5}), we have
\begin{eqnarray*}
BSO(T_n)&  =  & m_{1,2}\sqrt{\frac{1}{4}+1}+m_{2,2}\sqrt{\frac{1}{4}+\frac{1}{4}}+\sum_{(i,j)\in N_1}\!\sqrt{\frac{1}{i^2}+\frac{1}{j^2}}m_{i,j}\\
&  =  & \sqrt{5}\Big[1+\sum_{(i,j)\in N_1}\Big(1-\frac{i+j}{ij}\Big)m_{i,j}\Big]+\frac{\sqrt{2}}{2}\Big\{n-3+\sum_{(i,j)\in N_1}\Big[\frac{2(i+j)}{ij}-3\Big]m_{i,j}\Big\}\\
&     &+\sum_{(i,j)\in N_1}\sqrt{\frac{1}{i^2}+\frac{1}{j^2}}m_{i,j}\\
&  =  & \frac{\sqrt{2}}{2}(n-3)+\sqrt{5}+\sum_{(i,j)\in N_1}\Big[\sqrt{\frac{1}{i^2}+\frac{1}{j^2}}+({\sqrt{2}-\sqrt{5}})\frac{i+j}{ij}+\sqrt{5}-\frac{3\sqrt{2}}{2}\Big]m_{i,j}.
\end{eqnarray*}

Let $f(x,y)=\sqrt{\frac{1}{x^2}+\frac{1}{y^2}}+({\sqrt{2}-\sqrt{5}})\frac{x+y}{xy}+\sqrt{5}-\frac{3\sqrt{2}}{2}$, where $(x,y)\in N$, it is easy to see that $f(1,2)=0$, $f(2,2)=0$ and $f(x,y)> 0$ for $(x,y)\in N_1$. Therefore, $BSO(T_n)=\frac{\sqrt{2}}{2}(n-3)+\sqrt{5}$ if and only if $m_{i,j}=0$ for all $(i,j)\in N_1$. And this occurs if and only if $T_n\cong P_n$. Conversely, if $T_n\cong P_n$, by (\ref{eq5}), we obtain $$BSO(P_n)=2\sqrt{\frac{1}{4}+1}+(n-3)\sqrt{\frac{1}{4}+\frac{1}{4}}=\frac{\sqrt{2}}{2}(n-3)+\sqrt{5}.$$
Thus, we have $BSO(T_n)\geq BSO(P_n)$ with equality if and only if $T_n\cong P_n$.

Now, we consider the equality on the right side. Let $N_2=\Big\{(i,j)\in N:(i,j)\neq(1,1),(i,j)\neq(1,\Delta),(i,j)\neq(\Delta,\Delta)\Big\}$. Similar to the proof of the above, by equation (\ref{eq4}), we have
$$(\Delta+1)m_{1,\Delta}+2m_{\Delta,\Delta}=\Delta n-\sum_{(i,j)\in N_2}\Delta\frac{i+j}{ij}m_{i,j},$$
 and by equation (\ref{eq2}), we have
$$m_{1,\Delta}+m_{\Delta,\Delta}=n-1-\sum_{(i,j)\in N_2}m_{i,j}.$$
Then we obtain the following expression for $m_{1,\Delta}$ and $m_{\Delta,\Delta}$:
$$(\Delta-1)m_{1,\Delta}=(\Delta-2)n+2-\sum_{(i,j)\in N_2}\Big(\Delta\frac{i+j}{ij}-2\Big)m_{i,j},$$
$$(\Delta-1)m_{\Delta,\Delta}=n-(\Delta+1)+\sum_{(i,j)\in N_2}\Big(\Delta\frac{i+j}{ij}-(\Delta+1)\Big)m_{i,j}.$$
According to the expression (\ref{eq4}), we have
\begin{eqnarray*}
BSO(T_n) &  =  & m_{1,\Delta}\sqrt{\frac{1}{\Delta^2}+1}+m_{\Delta,\Delta}\sqrt{\frac{1}{\Delta^2}+\frac{1}{\Delta^2}}+\sum_{(i,j)\in N_2}\sqrt{\frac{1}{i^2}+\frac{1}{j^2}}m_{i,j}\\
&  =  & \frac{\sqrt{\Delta^2+1}}{\Delta(\Delta-1)}\Big[(\Delta\!-\!2)n+2-\sum_{(i,j)\in N_2}\Big(\Delta\frac{i+j}{ij}-2\Big)m_{i,j}\Big]\\
&     & +\frac{\sqrt{1+\Delta^2}}{\Delta(\Delta-1)}\Big[n\!-\!(\Delta+1)+\!\sum_{(i,j)\in N_2}\!\Big(\Delta\frac{i+j}{ij}\!-\!(\Delta+1)\Big)m_{i,j}\Big]\\
&     & +\!\sum_{(i,j)\in N_2}\!\sqrt{\frac{1}{i^2}+\frac{1}{j^2}}m_{i,j}\\
&  =  & \frac{(\Delta-2)n\sqrt{\Delta^2+1}+\sqrt{2}(n-\Delta-1)+2\sqrt{\Delta^2+1}}{\Delta(\Delta-1)}\\
&     & +\sum_{(i,j)\in N_2}\Big[\sqrt{\frac{1}{i^2}+\frac{1}{j^2}}+\frac{\sqrt{2}-\sqrt{\Delta^2+1}}{\Delta-1}\frac{i+j}{ij}+\frac{2\sqrt{\Delta^2+1}-\sqrt{2}(\Delta+1)}
{\Delta(\Delta-1)}\Big]m_{i,j}.
\end{eqnarray*}

Let $g(x,y)=\sqrt{\frac{1}{x^2}+\frac{1}{y^2}}+\frac{\sqrt{2}-\sqrt{\Delta^2+1}}{\Delta-1}\frac{x+y}{xy}+\frac{2\sqrt{\Delta^2+1}-\sqrt{2}(\Delta+1)}{\Delta(\Delta-1)}$,  where $(x,y)\in N$, it is easy to see that $f(1,\Delta)=0$, $f(\Delta,\Delta)=0$ and $f(x,y)< 0$ for $(x,y)\in N_2$. Therefore, $BSO(T_n)=\frac{(\Delta-2)n\sqrt{\Delta^2+1}+\sqrt{2}(n-\Delta-1)+2\sqrt{\Delta^2+1}}{\Delta(\Delta-1)}$ if and only if $m_{i,j}=0$ for all $(i,j)\in N_2$. And this occurs if and only if $n_2=n_3=\ldots=n_{\Delta-1}=0$.

Let $h(x)=\frac{(x-2)n\sqrt{x^2+1}+\sqrt{2}(n-x-1)+2\sqrt{x^2+1}}{x(x-1)}$. By derivative, we know that $h(x)$ is an increasing function for $[2, +\infty)$. Thus 
$$h(\Delta)\leq h(n-1)=\sqrt{1+(n-1)^2}.$$
Conversely, $BSO(K_{1,\,n-1})=\sqrt{1+(n-1)^2}$. Thus, we have $BSO(T_n)\leq BSO(K_{1,\,n-1})$ with equality if and only if $T_n\cong K_{1,\,n-1}$.

\end{proof}

Similar to the method used in Theorem \ref{th4,1}, we now give an upper bound on chemical trees without its proof.

\begin{theorem}\label{th4,2} 
Let $T_n$ is a chemical tree with $n$ vertices. If $n-2=0(mod\ 3)$, then
$$BSO(T_n)\leq \frac{2\sqrt{17}(n+1)+\sqrt{2}(n-5)}{12}$$
with equality if and only if $n_2=n_3=0$.
\end{theorem}

\small {

}

\end{document}